\documentclass[10pt]{amsart}

\usepackage[all]{xy}                        %

\CompileMatrices                            

\UseTips                                    

\input xypic
\usepackage[bookmarks=true]{hyperref}       

\usepackage{amssymb,latexsym,amsmath,amscd}
\usepackage{xspace}
\usepackage{enumerate}
\usepackage{graphicx}
\usepackage{vmargin}
\usepackage{todonotes}

\reversemarginpar

\theoremstyle{plain}
\newtheorem{theorem}{Theorem}[section]
\newtheorem*{theorem*}{Theorem}
\newtheorem{proposition}[theorem]{Proposition}

\newtheorem{lemma}[theorem]{Lemma}

\theoremstyle{definition}
\newtheorem{definition}[theorem]{Definition}

\newtheorem{remark}[theorem]{Remark}
\newtheorem{example}[theorem]{Example}

\newcommand{\enm}[1]{\ensuremath{#1}}          %

\newcommand{\cal}[1]{\mathcal{#1}}

\renewcommand{\bar}[1]{\overline{#1}}

\newcommand{\CC}{\enm{\mathbb{C}}}

\newcommand{\RR}{\enm{\mathbb{R}}}

\newcommand{\ZZ}{\enm{\mathbb{Z}}}

\newcommand{\PP}{\enm{\mathbb{P}}}

\newcommand{\HH}{\enm{\mathbb{H}}}

\newcommand{\Ee}{\enm{\cal{E}}}

\newcommand{\Gg}{\enm{\cal{G}}}

\newcommand{\Oo}{\enm{\cal{O}}}

\newcommand{\Qq}{\enm{\cal{Q}}}

\renewcommand{\phi}{\varphi}
\renewcommand{\theta}{\vartheta}
\renewcommand{\epsilon}{\varepsilon}

\begin{document}

\title[Surfaces with infinitely many twistor lines]{Algebraic surfaces with infinitely many twistor lines}

\author[A. Altavilla]{A. Altavilla${}^{\ddagger}$}\address{Altavilla Amedeo: Dipartimento di Ingegneria Industriale e Scienze
Matematiche, Universit\`a Politecnica delle Marche, Via Brecce Bianche, 60131,
Ancona, Italy} \email{altavilla@dipmat.univpm.it}

\author[E. Ballico]{E. Ballico${}^{\dagger}$}\address{} \email{}\address{Edoardo Ballico: Dipartimento Di Matematica, Universit\`a di Trento, Via Sommarive 14, 38123, Povo, Trento, Italy} \email{edoardo.ballico@unitn.it}

\thanks{${}^{\dagger,\ddagger}$GNSAGA of INdAM;
 ${}^{\dagger}$MIUR PRIN 2015 ``Geometria delle variet\`a algebriche'';
 ${}^{\ddagger}$SIR grant {\sl ``NEWHOLITE - New methods in holomorphic iteration''} n. RBSI14CFME and SIR grant {\sl AnHyC - Analytic aspects in complex and hypercomplex geometry} n. RBSI14DYEB. The first author was also financially supported by a INdAM fellowship ``{\em mensilit\'a
di borse di studio per l'estero a.a. 2018--2019}'' and wants to thanks the Clifford Research Group at Ghent University where this fellowship has been spent.}

\date{\today }


\begin{abstract} 
We prove that a reduced and irreducible algebraic surface in $\mathbb{CP}^{3}$ containing infinitely many twistor lines cannot have odd degree. Then, exploiting the theory of quaternionic slice regularity and the normalization map of a surface, we give constructive existence results for even degrees.
\end{abstract}
\maketitle

\section{Introduction and Main Results}
In this paper we study integral (i.e. reduced and irreducible) algebraic surfaces in $\CC\PP^{3}$ containing infinitely many twistor lines.
Let $\mathbb{HP}^{1}$ denote the \textit{left} quaternionic projective line. This manifold is diffeomorphic to the 4-sphere  $\mathbb{S}^{4}$.
A twistor line is then a fiber of the usual twistor fibration 
$$
\mathbb{CP}^{1}\to\mathbb{CP}^{3}\stackrel{\pi}{\to}\mathbb{HP}^{1}(\simeq\mathbb{S}^{4}),
$$
defined as
$$
\pi[z_{0},z_{1},z_{2},z_{3}]=[z_{0}+z_{1}j,z_{2}+z_{3}j],
$$
where $j\in\mathbb{H}$ is such that $ij=k$ and $(i,j,k)$ is the standard basis of imaginary units in $\mathbb{H}$.
Motivations to study this fibration come from its link with Riemmannian and complex geometry (see e.g.~\cite{lebrun}).

It is known (see e.g.~\cite{gensalsto}) that twistor lines can be identified with projective lines $\ell\subset\CC\PP^{3}$ such that $j(\ell)=\ell$,
where $j:\mathbb{CP}^{3}\to\mathbb{CP}^{3}$ is the fixed-point-free anti-holomorphic involution given by
$$
j[z_{0},z_{1},z_{2},z_{3}]\mapsto [-\bar z_{1}, \bar z_{0}, -\bar z_{3},\bar z_{2}].
$$
Moreover, the map $j$ induces (via Pl\"ucker embedding), a map (also called $j$), in the Grassmannian 
$Gr(2,4):=\{t_{1}t_{6}-t_{2}t_{3}+t_{4}t_{5}=0\}\subset\CC\PP^{5}$, defined as follows (see e.g. \cite[Section 3]{altavillaballico}):
\begin{equation}\label{mapjgrass}
j([t_{1},t_{2},t_{3},t_{4},t_{5},t_{6}])=[\overline{t_{1}},\overline{t_{5}},-\overline{t_{4}},-\overline{t_{3}},\overline{t_{2}},\overline{t_{6}}],
\end{equation}
and twistor lines can be identified then as points in $Gr(2,4)$ which are fixed by this map $j$.

The study of algebraic surfaces from the twistor projection point of view is somehow complete in the case of planes and quadrics~\cite{chirka,sv1},
but still partial in the case of cubics~\cite{armstrong,APS,sv2}. In a series of papers the authors have given general results
on this topic by exploiting analytic~\cite{altavilla,altavillasarfatti} (see also~\cite{gensalsto}) and algebraic~\cite{altavillaballico,altavillaballico2,ballico}
methods.

The goal of this paper is to use classical algebraic geometry and quaternionic slice regularity to
show that there are not odd degree integral surfaces containing infinitely many twistor lines and that
 for each even degree there exists at least one. We will prove this last statement by giving two methods of construction.
First of all, thanks to~\cite[arXiv version v1, Remark 14.5]{sv1}, an integral degree $d$ algebraic surface containing 
more than $d^{2}$ twistor lines has to be $j$-invariant, hence a surface containing infinitely many twistor
lines is $j$-invariant.
Surfaces with infinitely many lines are ruled and non-normal and so we will deal with this class.
However we will show with a simple argument that
 cones are not allowed. Given a ruled surface $Y$ we will recall its normalization map $u:\PP(\Ee)\to Y$, $\Ee$ being
a rank 2 vector bundle over a smooth curve $C$. Given such a vector bundle $\Ee$ and $L\subset\Ee$
a rank 1 subsheaf of maximal degree, we say that $\Ee$ has the property $\pounds$ if $L$ is the unique
rank 1 subsheaf of maximal degree (see Definition~\ref{pound}). Afterwards, in Section 2, we will recall some known fact on the stability of rank 2 vector bundles over a smooth curve and we will link them to the property $\pounds$. 

Using the results proven in Section 2, in Section 3, assuming that the surface $Y$ is integral and ruled by twistor lines, we are able to  prove that its normalization $\PP(\Ee)$ is such that $\Ee$ has not $\pounds$ and,
equivalently that $\Ee$ is semi-stable (see Theorem~\ref{x1}). As a direct consequence we obtain that no odd degree integral rational surface with infinitely many  twistor lines exists. More precisely, we prove the following.

\begin{proposition}\label{x2}
Let $Y\subset \CC\PP^3$ be an integral rational surface containing infinitely many twistor lines. Then $\deg (Y)$ is even
and $\CC\PP^1\times \CC\PP^1$ is the normalization of $Y$.
\end{proposition}

With a bit more effort we are able to remove the rationality hypothesis and to prove the following next result.

\begin{theorem}\label{bbb1}
Let $Y\subset \mathbb{CP}^{3}$ be an integral surface containing infinitely many twistor lines. Then $\deg (Y)$ is even. 
\end{theorem}

Afterwards we give two existence results for even degrees by showing constructive methods.
In the first result, using the theory of slice regularity (see~\cite{genstostru} for an overview in the quaternionic setting), and the results contained in~\cite{altavilla, altavillasarfatti, gensalsto}, we are able to solve the problem with the hypothesis of rationality. In fact a slice regular function is a quaternionic function of a quaternionic variable
$f:\Omega\subset\HH\to\HH$ whose restrictions to any complex plane $\CC_{p}=span_{\RR}\langle 1, p\rangle\subset \HH$, such that $p^{2}=-1$, are holomorphic functions. 
Fixed any orthonormal basis $\{1,i,j,k\}\subset\HH$ it is possible to split a slice regular function $f$
as $f=g+hj$, where $g|_{\CC_{i}}, h|_{\CC_{i}}$ are complex holomorphic function of a $\CC_{i}$-variable.
The key result is then that any slice regular function $f$ can be \textit{lifted} (via $\pi^{-1}$), to
a holomorphic function $\tilde{f}:\Qq\subset\CC\PP^{3}\to\CC\PP^{3}$, where $\Qq\cong \CC\PP^{1}\times\CC\PP^{1}$, and the expression of $\tilde{f}$ is explicitly given in terms of the
splitting $f=g+hj$. With this tool, the precise result that we are able to prove is the following.

\begin{proposition}\label{x3} 
For each even integer $d\ge 2$ there is a rational degree $d$ ruled surface $Y\subset \CC\PP^3$ containing infinitely many twistor lines.
\end{proposition}

The last (constructive) result of the paper states that it is possible to select a smooth curve $C$ to construct an integral surface $Y$ with infinitely many twistor lines, such that  $u:\PP(\Ee)\to Y$ is its normalization with $\Ee$ is a rank 2 vector bundle over a smooth curve $C$. 
The precise statement is the following.

\begin{theorem}\label{ox1}
Let $C$ be a smooth and connected complex projective curve defined over $\RR$ and with $C(\RR )\ne \emptyset$. Fix an integer $d_0$. Then there is an integer $d\ge d_0$
and a degree $d$ integral surface $Y\subset \CC\PP^3$ such that $Y$ contains infinitely many twistor lines and the normalization of $Y$ is a $\CC\PP^1$-bundle over $C$.
\end{theorem}

The proof of this last theorem is rather constructive and it is exploited in the last example to generate a class of integral ruled surfaces of even degree each of them containing infinitely many twistor lines.


%
%
%
%
%

\section{Preliminary results}

The main reference for this section is~\cite[V]{h}.
Let $C$ be a smooth and connected complex projective curve of genus $g\ge 0$, $\Ee$ be a rank $2$ holomorphic
vector bundle on $C$ and $L\subset \Ee$ be a rank $1$ subsheaf of $\Ee$ with maximal degree. As in~\cite[V.2]{h} or~\cite{ln}, but with opposite sign, set $s(\Ee):= 2\deg (L)-\deg (\Ee ) $. Note that for any line bundle $R$ on $C$, the line subbundle $L\otimes R$ of
$\Ee
\otimes R$ is a rank $1$ subsheaf of $\Ee\otimes R$ with maximal degree and hence $s(\Ee )=s(\Ee \otimes R)$. Thus the integer
$s (\Ee)$ depends only from the isomorphism classes of the $\CC\PP^1$-bundle $\PP (\Ee)$. We have $s(\Ee )\equiv \deg (\Ee
)\pmod{2}$. Thus the parity classes of the integer $\deg (\Ee)$ and $s(\Ee )$ are constant in connected families of rank $2$
vector bundles on $C$ and the parity class of $s(\Ee )$ is a deformation invariant for the smooth surface $\PP (\Ee )$.

\begin{definition}\label{pound}
Let $C$ be a smooth and connected complex projective curve of genus $g\ge 0$, $\Ee$ be a rank $2$ holomorphic
vector bundle on $C$ and $L\subset \Ee$ be a rank $1$ subsheaf of $\Ee$ with maximal degree.
We say
that
$\Ee$ \textit{has (the property)}
$\pounds$ if
$L$ is the unique rank $1$ subsheaf of $\Ee$ with maximal degree. 
\end{definition}

Thanks to previous considerations, for any line bundle $R$ on $C$ the vector bundle $\Ee$ has
$\pounds$ if and only if $\Ee \otimes R$ has $\pounds$.
 Hence we are allowed to say if a $\CC\PP^1$-bundle  $\PP (\Ee)$ has
$\pounds$ or not. 
\begin{remark}\label{Hir}
If the genus $g$ of $C$ is equal to zero, then $s(\Ee)=s$ if and only if $\PP (\Ee )\cong F_{-s}:= \PP(\Oo _{\PP^1}\otimes
\Oo_{\PP^1}(s))$, i.e. is the Hirzebruch surface with invariant $-s$ (see~\cite[V 2.13, 2.14]{h}).
\end{remark}

By the
definition of stability and semistability for rank $2$ vector bundles we see that $s(\Ee)<0$ (resp. $s(\Ee)\le 0$) if and only if $\Ee$
is stable (resp. semistable), moreover, $s (\Ee )=0$ if and only if $\Ee$ is strictly semistable, i.e. it is semistable but not stable (see~\cite[V Exercise 2.8]{h} and also~\cite{ln}).

\begin{lemma}\label{x1.001}
Let $C$ be a smooth curve and  $\Ee$ be a rank $2$ vector bundle on $C$ without $\pounds$. Then $\Ee$ is semistable.
\end{lemma}

\begin{proof}
Let $\Gg$ be a rank $2$ vector bundle on $C$ which is not semistable. It is sufficient to prove that $\Gg$ has $\pounds$.  Let $L\subset \Gg$ be a maximal degree rank $1$ subsheaf. Since
$L$ has maximal degree, then $\Gg/L$ has no torsion, i.e. (since $C$ is a smooth curve) it is a line bundle. We have the following exact sequence
\begin{equation*}
0 \to L \to \Gg\stackrel{v}{\to} \Gg/L\to 0.
\end{equation*}
Let $M$ be any maximal degree rank $1$ subsheaf of $\Gg$. We have $\deg (M) = \deg (L)$. Since $\Gg$ is not semistable, we have $\deg (L)>\deg (\Gg/L)$ and so
$\mathrm{Hom}(M,\Gg/L) =0$. Thus $v_{|M} \equiv 0$, i.e. $M\subseteq L$. Since $\deg (M)=\deg (L)$, we get $M=L$, i.e. $\Gg$ has $\pounds$.
\end{proof}

%

Let $C$ be any smooth and connected projective curve of genus $g\ge 2$. See~\cite{ln} for a huge number of examples of stable
rank
$2$ vector bundles on $C$ with $\pounds$ or without $\pounds$.

It is proved in~\cite{ln} that if $\Ee$ is a
general rank $2$
stable vector bundle on $C$ with degree $d$, then the integer $s(\Ee )$ is the only integer in $\{-g,1-g\}$ which is $\equiv
d\mod{2}$. If $s(\Ee )=-g$, then $\Ee$ has $\infty ^1$ maximal degree rank $1$ subbundles and hence it has not $\pounds$.
If $s(\Ee )=1-g$ and $\Ee$ is general, then $\Ee$ has exactly $2^g$ maximal degree $1$ subbundles (a result discovered by C. 
Segre in 1889).

We recall the following well-known observation which characterizes the property $\pounds$ in
the case of strictly semi-stability.

\begin{lemma}\label{decomp}
Assume $s(\Ee )=0$. $\Ee$ has $\pounds$ if and only if $\Ee$ is indecomposable. If $\Ee$ is decomposable, then either $\Ee
\cong L^{\oplus 2}$ for some line bundle $L$ (and in this case $\Ee$ has $\infty ^1$ rank $1$ subsheaves with maximal degree) 
or $\Ee \cong L\oplus M$ with $L, M$ line bundles, $\deg (L)=\deg (M)$ and $L\ncong M$ (and in this  case $L$ and $M$ are the only
line subbundles of $\Ee$ with maximal degree).
\end{lemma}

\begin{proof}
Assume $\Ee$ decomposable, say $\Ee \cong L_1\oplus L_2$ with $L_1$ and $L_2$ line bundles on $C$ with $\deg (L_2)\ge \\deg
(L_1)$.
Since  $s(\Ee )=0$ we in particular have $s(\Ee )\le 0$ and $\deg (L_2)=\deg (L_1)$. 
Let
$\pi _i:\Ee \to L_i$ denote the projections. Let $L$ be a maximal degree rank $1$ subsheaf $L\subset L_1\oplus L_2$, then
there is
$i\in \{1,2\}$ with $\pi _{i|L} \ne 0$. Hence $\deg (L)\le \deg (L_i)$ and equality holds if and only if $\pi _i$ induces
an isomorphism $L\to L_i$. The maximality property of $\deg (L)$ implies $\deg (L) =\deg (L_i)$. We get the second assertion
of the lemma and the ``only if '' part of the first assertion.

Now assume that $\pounds$ fails and take rank $1$ subsheaves $L, M$ of $\Ee$ with maximal
degree. $M$ and $L$ may be isomorphic as abstract line bundles, but they are supposed to be different subsheaves of $\Ee$.
Since $\deg (L)= \deg (\Ee )/2 =\deg (M)$, we have $L\nsubseteq M$ and $M\nsubseteq L$. Hence the map $f: L\oplus M\to \Ee$
induced by the inclusions $L\hookrightarrow \Ee$ and $R\hookrightarrow \Ee$ have generic rank $2$. Since $f$ has generic rank
$2$
and $L\oplus M$ is a rank $2$ vector bundle, $f$ is injective. Since $s(\Ee )=0$, we have $\deg (L\oplus M)=\deg (\Ee)$. Thus
$f$ is an isomorphism, concluding the proof of the lemma.
\end{proof}

\section{Surfaces with infinitely many twistor lines}

Let now $Y\subset \CC\PP^3$ be an integral ruled projective surface of degree $>1$ and let $u: X\to Y$ denote the normalization map. Assume that $Y$ is not a cone. Then, $X$ is a $\CC\PP^1$-bundle on a smooth curve $C$, i.e. there is a rank $2$ vector
bundle $\Ee$ on $C$ such $X\cong \PP (\Ee)$.
Let $v: \PP (\Ee )\to C$ denote the map with $\CC\PP^1$ as fibers.
In particular $u$ sends each fiber of the ruling $v : \PP (\Ee) \to C$ to a line of $\mathbb {CP}^{3}$. 
 The map $v$ is a locally trivial fibration (both in the Zariski and the euclidean to topology), and the curve $C$ may be obtained in the following way. Fix a general hyperplane $H\subset \CC\PP^3$. Since $H$ is general, $H\cap Y$ is an integral plane curve. The curve $C$ is the normalization of the curve $H\cap Y$.

In this section we discuss the existence or the non-existence of integral degree $d$ surfaces $Y\subset \CC\PP^3$, $d\ge 3$,
containing infinitely many twistor lines. As said in the introduction, any integral degree $d$ surface $Y\subset \CC\PP^3$ containing at least $d^2+1$
twistor lines is $j$-invariant. Hence, if $Y$ contains infinitely twistor lines, then $j(Y)=Y$. Since any two twistor lines are disjoint and a cone has only finitely many curves not through the vertex, we may
exclude cones (moreover no smooth surface of degree $>2$ contains infinitely many lines, therefore we need to allow singular surfaces).
Hence, our peculiar situation fits well in the construction described at the beginning of the section.

A first interesting result is the following, concerning the property $\pounds$ (and the semistability).

\begin{theorem}\label{x1}
Let $Y\subset \CC\PP^3$ be an integral surface containing infinitely many twistor lines. Let $\PP (\Ee )$, $\Ee$ a rank $2$
vector bundle on a smooth curve $C$, be the normalization of $Y$. Then $\Ee$ has not $\pounds$ and in particular, by Lemma~\ref{x1.001}, it is
semistable.
\end{theorem}

\begin{proof}
We know that $j(Y)=Y$, that $Y$ is not a cone and that $Y$ contains infinitely many twistor lines appearing as lines of
the ruling. Let $u: \PP(\Ee )\to Y$ denote the normalization map. Assume that $\Ee$ has $\pounds$ and take
$T\subset \PP(\Ee)$ the section of the ruling $v: \PP (\Ee )\to C$ associated to the unique rank $1$ line subbundle L of $\Ee$ with maximal degree. 
The maximality of the integer $\deg (L)$ implies that $L$ is a rank $1$ subbundle of $\Ee$, i.e. that $\Ee/L$ is a line bundle on $C$ and that $\deg (\Ee/L)$ is minimal degree of a line bundle $M$ such that there
is a surjective map $f: \Ee \to M$. Surjective maps $f: \Ee \to M$ (or, equivalently, embedding of rank $1$ subbundles $R\hookrightarrow \Ee$) corresponds to sections of the ruling $v$ (\cite[Proposition V.2.6]{h}). 
Since $u$ is the
normalization map, $u$ is finite. 
Thus $L\subset \PP(\Ee)$ corresponds to a minimal degree curve $D\subset Y$ which can be seen exactly as the image of $T$, i.e. $D=u(T)$.
Since $T$ intersects each fiber of $Y$, $D$ intersects
each line of the ruling of $Y$. Each section of $v$ different from $T$ has as image in $\CC\PP^3$ a curve
of degree $>\deg (D)$. Thus, since $deg(j(D))=deg(D)$, then the section giving $j(D)$ equals the one giving $D$ and hence $j(D)=D$. Fix a twistor line $\ell$ of the ruling and take $z \in D\cap \ell$ ($z$ exists, because each
fiber of $v$ meets $T$). Since
$\ell$ is a twistor line, we have $j(z)\in \ell$. Since $j(D)=D$, we have $j(z)\in D$. Since $j: \CC\PP^3\to \CC\PP^3$ has no
fixed point, then $\ell$ contains at least two different points of $D$. Hence the fiber of $v$ over $v(u^{-1}(\ell))$ meets $T$ at at least two
different points, a contradiction.
\end{proof}

As a corollary of the previous result, we obtain immediately the following.

\begin{proof}[Proof of Proposition~\ref{x2}]
Set $d:= \deg (Y)$. By Theorem~\ref{x1} the normalization of $Y$ is associated to a degree $d$ rank $2$ vector bundle on $C$.
Since $Y$ is rational, the genus $g$ of $C$ is not positive, therefore $C\cong \CC\PP^{1}$. 
But then the normalization of $Y$ is the Hirzebruch surface with invariant
$s=s(\Ee )=0$ (see Remark~\ref{Hir}).
Thanks to Theorem~\ref{x1} and Lemma~\ref{decomp} two line bundles on $\CC\PP^{1}$ of the same degree are isomorphic,
thus $\Ee \cong L^{\oplus 2}$ with $L\cong \Oo _{\CC\PP^1}(d/2)$. Therefore $d$ is even and $\PP (L^{\otimes 2})
\cong \PP (\Oo _{\CC\PP^1}^{\oplus 2} )\cong \CC\PP^1\times \CC\PP^1$.
\end{proof}

To remove the hypothesis of rationality in Proposition~\ref{x2} we need the following Lemma whose proof is in the same spirit of the one of Theorem~\ref{x1}.

\begin{lemma}\label{jD}
Let $Y\subset \CC\PP^3$ be an integral surface containing infinitely many twistor lines. Let $\PP (\Ee )$, $\Ee$ a rank $2$
vector bundle on a smooth curve $C$, be the normalization of $Y$. 
Let $L\subset \Ee$ be a line bundle with maximal degree and $D\subset Y$ be a minimal degree curve which is the image of a section $T$ of $v$ corresponding to $L$. Then $j(D)\neq D$.
\end{lemma}
\begin{proof}
Assume $j(D)=D$. Thus $j$ induces an anti-holomorphic involution of $D$. Fix a twistor line $\ell$ of the ruling and take $z \in D\cap \ell$. Since
$\ell$ is a twistor line, we have $j(z)\in \ell$. Since $j(D)=D$, we have $j(z)\in D$. Since $j: \CC\PP^3\to \CC\PP^3$ has no
fixed point, $\ell$ contains at least two different points of $D$. Hence the fiber $v$ over $v(u^{-1}(\ell))$ meets $D$ at at least two
different points, a contradiction.
\end{proof}

\begin{proof}[Proof of Theorem \ref{bbb1}:]
Since any plane contains exactly one twistor line, we may assume $d\ge 3$. Let $u: X=\PP(\Ee)\to Y$ denote the normalization map. 
We recall that there is a smooth and connected curve $C$ and a rank $2$ vector bundle $\Ee$ on $C$ such that $X=\PP(\Ee)$ and $u$ sends each fiber of the ruling $v : \PP (\Ee) \to C$ to a line of $\mathbb {CP}^{3}$. 

Let $L\subset \Ee$ be a line bundle with maximal degree. 
As explained in the proof of Theorem~\ref{x1}, $L\subset \PP(\Ee)$ corresponds to a minimal degree curve $D\subset Y$ which is the image of a section of $v$. Since $j(Y)=Y$, $\deg (j(D)) =\deg (D)$ and (by Lemma~\ref{jD}) $j(D)\neq D$, then $j(D)$ corresponds to a maximal degree line subbundle $R\subset \Ee$ with $\deg ({R}) =\deg (L)$ and $R \neq L$.

%

Since $j(D)\ne D$ and $D$ is an integral curve, the set $S:= j(D)\cap D$ is finite. Note that $j(S)=S$. Since the anti-holomorphic involution $j$ has no base points then $b:= |S|$ is even. Moreover, $j_{|D} : D\to j(D)$ is a bijection and hence $j$ induces an anti-holomorphic involution $\hat{\j}: X\to X$.

Now we prove that $d$ is even. Let $T_1\subset X$ and $T_2\subset X$ be the sections of $v$ such that $u(T_1) =D$ and $u(T_2) =j(D)$. Since $u_{|T_1}: T_1\to D$
and $u_{|T_2}: T_2\to j(D)$ are bijections, $\hat{\j}$ induces a bijection $T_1\to T_2$ and $S':= T_1\cap T_2$ has cardinality $b$. For any divisors $A, B$ on $X$
let $A\cdot B$ denote the intersection product in the Chow ring of $X$, or equivalently, the cup product $H^2(X,\CC )\times H^2(X,\CC) \to H^4(X,\CC) \cong \CC$. We have $A\cdot B\in \ZZ$. For $p\in S'$ (if any) let $c_p$ be the degree of the connected component containing $p$ of the zero-dimensional scheme $T_1\cap T_2$ (scheme-theoretic intersection). Since $\hat{\j}$ is an anti-holomorphic isomorphism, we have $c_{j({p})} =c_p$ for all $p\in S'$. Since $T_1\cap T_2$ is finite, we have $T_1\cdot T_2 =\sum _{p\in S'} c_p$. Decomposing $S'$ into the disjoint union of pairs $\{o,\hat{\j}(o)\}$ we get that $T_1\cdot T_2$ is an even non-negative integer. Since $Y$ is not a cone, there is an ample and base point free
line bundle $\Oo _X(1)$ on $X$ such that $u$ is induced by a $4$-dimensional linear subspace of $H^0(\Oo _X(1))$ and $d =\Oo _X(1)\cdot \Oo _X(1)$. We have $\mathrm{Pic}(X) \cong v ^\ast (\mathrm{Pic}({C}))\oplus \ZZ T_1$ (\cite[Proposition V.2.3]{h}). Write $\sim$ for the numerical equivalence of divisors and line bundles on $X$: by definition, two fibers $F$ and $F'$ are equivalent if $A\cdot F =A\cdot F'$ for each divisor $A$.
Let $F$ denote the numerical equivalence class of a fiber of $v$. Since two different fibers of $v$ are disjoint, we have $F\cdot F =0$. For any degree $x$ line bundle $A$ on $C$ we have $v ^\ast (A) \sim xF$. Since
$Y$ is ruled by lines, we have $\Oo _X(1)\cdot F = 1$. Thus there is an integer $x$ such that $\Oo _X(1) \sim T_1+xF$. Since $T_2$ is a section of $v$, there is an integer $y$ such that $T_2 \sim T_1+yF$. We have $\deg (D) =\Oo _X(1)\cdot T_{1}= (T_1+xF) \cdot T_1 = T_1\cdot T_1 +x$ and $\deg (j(D)) =\Oo _X(1)\cdot T_{2} = (T_1+xF)\cdot (T_1+yF) =T_1\cdot T_1 +x+y$. Since $\deg (j(D)) =\deg (D)$, we have $y=0$. Thus $T_1\cdot T_1 =T_1\cdot T_2$. Hence $T_1\cdot T_1$ is an even non-negative integer.
We have $d =\Oo _X(1)\cdot \Oo _X(1) = (T_1+xF)\cdot (T_1+xF) = T_1\cdot T_1 +2x$. Since $T_1\cdot T_1$ is even, then $d$ is even.
\end{proof}


We now provide two different methods of constructing examples of integral surfaces with infinitely many twistor lines.
We begin with Proposition~\ref{x3}. As said in the introduction, the theory of quaternionic
slice regularity can be exploited, in this case, to prove the mentioned result. In fact, as it is
explained in~\cite{altavilla,altavillasarfatti,gensalsto}, any slice regular function $f:\Omega\subset\HH\to\HH$
can be lifted, with an explicit parametrization, to a holomorphic curve $\tilde f:\CC\PP^{1}\times\CC\PP^{1}\to\CC\PP^{3}$. This geometric construction is the core of the proof.

\begin{proof}[Proof of Proposition~\ref{x3}]
For any even degree $d$ it is possible to construct a rational ruled surface $Y\subset \CC\PP^3$
parameterized by the twistor lift $\tilde{f}$ of a slice regular function $f$~\cite{altavilla, altavillasarfatti,gensalsto}, i.e.
$$\tilde{f}: \CC\PP^1\times  \CC\PP^1\to Y,\quad ([s,u],[1,v])\mapsto [s,u,sg(v)-u\hat{h}(v),sh(v)+u\hat{g}(v)],$$
where $g,\hat{g}, h$ and $\hat{h}$ are holomorphic functions defined on $\CC$. As it is explained in Remark 4.9 of~\cite{altavillasarfatti}, if $\hat{g}(v)=\overline{g(\bar v)}$ and $\hat{h}(v)=\overline{h(\bar v)}$, then, 
$\deg(Y)$ is even and $Y$ contains infinitely many twistor fibers (namely the fibers over $f(\RR)$). Moreover, suitably choosing these function, it is possible to construct a birational morphism
between $\CC\PP^{1}\times\CC\PP^{1}$ and $Y$.\end{proof}

We now pass to the last part in which we prove Theorem~\ref{ox1}.
For what concern this last part, we notice that the map $j$ defined in Equation~\eqref{mapjgrass} can be decomposed as $j = ^{\bar{\ }}\circ \sigma = \sigma \circ ^{\bar{\ }}$,
where $^{\bar{\ }}, \sigma:\CC\PP^{5}\to\CC\PP^{5}$ are defined as
\begin{align*}
\sigma ([t_1:t_2:t_3:t_4:t_5:t_6]) & = [t_1:t_5:-t_4:-t_3:t_2:t_6]\\
\bar{[t_1:t_2:t_3:t_4:t_5:t_6]} & = [\bar{t}_1:\bar{t}_2:\bar{t}_3:\bar{t}_4:\bar{t}_5:\bar{t}_6]
\end{align*}

We recall that $\CC\PP^1$ is defined over $\RR$ with $\RR\PP^1$ as its real points. We recall that if $g>0$ there are infinitely many pairwise non-isomorphic smooth and connected complex projective curve $C$ of genus $g$, defined over $\RR$ and with $C(\RR )\ne \emptyset$ \cite{gh, sep}.

\begin{proof}[Proof of Theorem~\ref{ox1}:]
Set $F: = \{t_2-t_5 = t_3+t_4=0\}\subset \CC\PP^5$ and $E:=Gr (2,4)\cap F$. Note that $\sigma _{|F}$ is the identity map. Take homogeneous coordinates $t_1,t_5,t_4,t_6$ on $F \cong \CC \PP^3$.
 Note that $E$ is the smooth quadric surface of $F$ with $t_1t_6 -t_5^2-t_4^2=0$ as its equation, i.e. over $\RR$ the quadric $E$ has signature $(1,3)$. So $E$ has many real points, but it is not projectively isomorphic to $\RR\PP^1\times \RR\PP^1$. Set $F':= F\setminus \{t_6=0\}\cong \CC ^3$

Let $C$ be a smooth and geometrically connected projective curve defined over $\RR$ and with $C(\RR )\ne \emptyset$. In this case $C(\RR )$ is topologically isomorphic to the disjoint union of $k$ circles, with $1\le k\le g+1$ (we only use that $C(\RR )$ is infinite and hence it is dense in  $C(\CC )$ for the Zariski topology).
Fix $p\in C  (\RR )$ and set $C':= C\setminus \{p\}$. $C'$ is an affine and connected rational curve defined on $\RR$. Thus there are non-constant algebraic maps $f_4: C'\to \CC$
and $f_5: C'\to \CC$ defined over $\RR$. Set $f_1:= f_4^2+f_5^2$. 

The map $(f_1,f_4,f_5): C'\to \CC ^3$ is defined over $\RR$ and it maps $C'$ into $F'\cap E$. Since $C$ is a smooth projective curve, $(f_1,f_4,f_5)$, extends in a unique way to a regular map $\psi : C\to F$. As $(f_1,f_4,f_5)$ is defined over $\RR$, the uniqueness of the
extension $\psi$ gives that $\psi$ commutes with the complex conjugation. The image of $C'$ is contained in $F'\cap E$, therefore we have $D:= \psi ({C}) \subset E\subset Gr(2,4)$ and because $f_5$ is not constant, $D$ is an integral projective curve. Since $C$ and $\psi$ are defined over $\RR$, then $D$ is defined over $\RR$. 
We now state three claims which lead us to the thesis.

\quad \emph{Claim 1:} We may find $f_4$ and $f_5$ such that $\psi$ is birational onto $D$ (i.e.: $C$ is the normalization of $D$).

\quad \emph{Claim 2:} $j(D)=D$.

\quad \emph{Claim 3:} Let $Y$ be the integral surface in $\CC\PP^{3}$ defined by $Y=\cup_{p\in D}\ell_{p}$, where $\ell_{p}$ is the line in $\CC\PP^{3}$ corresponding to $p\in Gr(2,4)$. Then $Y$ contains infinitely many twistor line.

\quad \emph{Proof of Claim 1:} Since $C$ is compact, then $\psi$ is a proper map. Moreover, since $f_5$ is a non-constant algebraic map, then $f_5^2$ is a proper map deleting finitely many
points of $\CC$, i.e. there is a finite set $S\subset \CC$ such that, taking $C'':= \{(f_5^2)^{-1}(\CC \setminus S)\}$, $f_5^2$ induces a proper non-constant map $u': C'' \to \CC\setminus S$. Since $C''$
is an irreducible (affine) curve, the differential of this map vanishes only at finitely many points. Increasing if necessary the finite set $S$ we may assume that $u'$ has everywhere
non-zero differential. Fix any $a\in \CC \setminus S$ and set $S':= f_{5}^{-1}(a)$. $S'$ is a non-empty finite set. To prove Claim 1 it is sufficient to take $f_5$ such that $f_5^2(b) \ne
f_5^2(b')$ for all $b, b'\in S'$ such that $b\ne b'$ and to repeat the same argument for $f_{4}$.

\quad \emph{Proof of Claim 2:} Since $C$ and $\psi$ are defined over $\RR$, $D$ is defined over $\RR$, i.e. complex conjugation induces a real analytic isomorphism
between $D$ and itself. In particular complex conjugation induces a bijection of $D$. Thus to prove Claim 2 it is sufficient to prove that $\sigma (D)=D$. But we notice that
$\sigma _{|F}$ is the identity map and so conclude the proof of Claim 2.

\quad \emph{Proof of Claim 3:} Fix $a\in D(\RR )\subset Gr (2,4)$ and let $\ell_{a}\subset Y$ be the the line associated to $a$. Since $a\in D(\RR)$, $\ell_{a}$ is defined over
$\RR$ and hence the complex conjugation sends $\ell_{a}$ into itself. Thus it is sufficient to prove that $\sigma (\ell_{a})=\ell_{a}$, where $\sigma: \CC\PP^3\to \CC\PP^3$ is the holomorphic involution $\sigma$ defined before on $\CC \PP^5$. 
We have $t_1 = z_0\wedge z_1$, $t_2 = z_0\wedge z_2$, $t_3 = z_0\wedge z_3$, $t_4 = z_1\wedge z_2$, $t_5 = z_1\wedge z_3$ and $t_6 = z_2\wedge z_3$. 
We obtain 
\begin{align*}
\sigma (t_1)&=\sigma (z_0)\wedge \sigma (z_1) = (-z_1)\wedge (z_0) = z_0\wedge z_1 = t_1\\
\sigma (t_2)&=\sigma (z_0)\wedge \sigma (z_2) = (-z_1)\wedge (-z_3) = z_1\wedge z_3 =t_5 \\
\sigma (t_3)&=\sigma (z_0)\wedge \sigma (z_3) = (-z_1)\wedge z_2 = -t_4 \\
\sigma (t_4)&=\sigma (z_1)\wedge \sigma (z_2) = z_0\wedge (-z_3) = -t_3 \\
\sigma (t_5)&=\sigma (z_1)\wedge \sigma (z_3) = z_0\wedge z_2 = t_2\\
\sigma (t_6)&=\sigma (z_2)\wedge \sigma (z_3) = (-z_3)\wedge z_2 = z_2\wedge z_3 =t_6,
\end{align*}
concluding the proof of Claim 3.

After having proven these three claims, to conclude the proof we only need to observe that we may take $f_5$ such that the map $f_5^2: C' \to \CC$  has degree $\ge d_0$.

\end{proof}

The proof of Theorem~\ref{ox1} allows us to give several examples, all of them with $Y$ with even degree, as explained in the following example.
\begin{example}
In the set-up of Theorem~\ref{ox1} take $F =\CC\PP^3$ with homogeneous coordinates $t_1,t_4,t_5, t_{6}$.
Let $E\subset F$ be the smooth quadric surface with $t_1t_6= t_5^2+t_4^2$ as its equation. All integral projective curves
$D\subset E$ defined over $\RR$ gives examples of surfaces $Y\subset \PP^3$ with infinitely many twistor lines. Take as $D$ a
complete intersection of $E$ with a smooth quadric surface defined over $\RR$. We get the existence of a degree $4$ elliptic
ruled surface $Y\subset \CC\PP^3$ with infinitely many twistor lines. Taking general intersections of $E$ with a degree $t\ge 3$
hypersurface of $F$ defined over $\RR$ we find smooth $C$ with degree $2t$ and genus $t^2-2t+1 $ (adjunction formula).
This construction only gives ruled surfaces $Y$ of even degree (even allowing singular curves $D$), because $E$ contains only
even degree curves defined over $\RR$.
\end{example}

%
%
%
%
%

\end{document}